\documentclass[10pt]{amsart}
\usepackage{graphicx}
\usepackage{color}
\usepackage{latexsym}
\usepackage{amssymb}                
\usepackage{amsmath}
\usepackage{amsthm}
\usepackage{amscd}
\usepackage{enumerate}
\usepackage{verbatim}
\usepackage{xy}
\xyoption{all}
\theoremstyle{plain}
\newtheorem{thm}{Theorem}
\newtheorem{lem}{Lemma}

\newtheorem{prop}{Proposition}
\newtheorem{cor}{Corollary} 
  
\theoremstyle{remark}

\numberwithin{equation}{section}
\def\ens{\ensuremath}                                     
\newcommand\mtb[1]{\ens{\mathbb{#1}}}                     
	   \newcommand{\Q}{\mtb{Q}}   	\newcommand{\R}{\mtb{R}}

\newcommand\mtc[1]{\ens{\mathcal{#1}}}           
                	
\newcommand{\del}{\ens{\delta}}  \newcommand{\eps}{\ens{\varepsilon}}
\newcommand{\gam}{\ens{\gamma}}	\newcommand{\lam}{\ens{\lambda}}
\newcommand\bld[1]{\ens{\boldsymbol{#1}}} 
\newcommand{\bfx}{\ens{\bld x}}

\newcommand{\abs}[1]{\left|#1\right|}

\newcommand\hl[1]{{\center \color{red} \ens{\bs\bs\bs} \\}} 
\newcommand\hsp[1]{\mbox{}\hspace{#1mm}} 
\newcommand\hspm[1]{\mbox{}\hspace{-#1mm}} 
\newcommand\vsp[1]{\par \vspace{#1mm}} 
\def\bs{{\bigstar}}                     
\def\ui{\ens{[0,1]}}                   
\newcommand\mbc[1]{}                                                    

\renewcommand{\v}{\vsp}

\newcommand{\h}{\hsp}
\newcommand{\Dst}{\displaystyle}

\def\H{\ens{\mtc H}}  \def\HP{\ens{\mtc H^{+}}} \def\hp{\HP} 
\def\F{\ens{\mtc F}} 
 
\def\C{\ens{\mtc C}} 
\title[A dichotomy for the stability of arithmetic progressions]{A Dichotomy for the
stability of \\ arithmetic progressions}
\author[M. Boshernitzan]{Michael Boshernitzan}
\address{Department of Mathematics, Rice University, Houston, TX~77005, USA}
\email{michael@rice.edu}
\thanks{The authors were supported in part by NSF Grants DMS-1102298, DMS-1004372}
\author[J. Chaika]{Jon Chaika}
\address{Department of Mathematics, University of Chicago, 5734 S.\ University Avenue, 
Chicago, IL 60637, USA}
\email{jonchaika@gmail.com}
\date{March 2013}
\begin{document}
\maketitle
\begin{abstract}
Let $\H$ stand for the set of homeomorphisms\, $\phi\colon\!\ui\to\ui$.  
We prove the following dichotomy for Borel subsets $A\subset \ui$:
\begin{itemize}
\item either there exists a homeomorphism $\phi\in\H$ such that the image
$\phi(A)$  contains no $3$-term arithmetic progressions;
\item or, for every $\phi\in \H$, the image $\phi(A)$ contains 
arithmetic progressions of arbitrary finite length.
\end{itemize}
\v1
\noindent In fact, we show that the first alternative holds if and only if the set $A$ is meager
(a countable union of nowhere dense sets).
\end{abstract}
\section{Definitions}
Let $\R$, $\Q$ denote the sets of real and rational numbers, respectively.
By an AP (arithmetic progression)  we mean a finite strictly increasing sequence 
in $\R$ of the form $\bfx=(x+kd)_{k=0}^{n-1}$,  with  $d>0$ and $n\geq3$. 
The convention is sometimes abused by identifying the sequence $\bfx$ with the set 
of its elements. An AP is completely determined by its first term\, $x=\min \bfx$, its
length $n=\abs{\bfx}$ and its step (difference) $d>0$.

A subset $S\subset\R$ is called  FAP (free of APs) it it does not contain $3$-term APs.

A subset $S\subset\R$  is called RAP  (rich in APs)  if it contains APs of arbitrary large
finite length.

Denote by $\H$ the set of homeomorphisms $\phi\colon\ui\to\ui$ of the unit interval.
The result presented in the abstract can be restated as follows.
\begin{thm}\label{thm:central}
Let  $S\subset \ui$ be a Borel subset. Then exactly one of 
the following two assertions holds:
\begin{enumerate}
\item[\em(1)] (either) there exists a $\phi\in\H$ such that $\phi(S)$ is FAP;
\item[\em(2)] (or)  $\phi(S)$ is RAP for every $\phi\in\H$.
\end{enumerate}
Moreover, {\em(1)} holds if and only if $S$ is meager.
\end{thm}
Recall some basic relevant definitions. Let $S\subset\R$. A set
$S$ is called {\em nowhere dense} if its closure $\bar S\subset\R$ has empty interior.
$S$ is called {\em meager} (a set of first category), if it is a countable union 
of nowhere dense sets. 
$S$ is called {\em residual}, or {\em co-meager}, if $\R\!\setminus\!S$ is meager;
$S$ is called {\em residual} in a subinterval $X\subseteq\R$ if the complement 
$X\!\setminus\! S$  is meager.
Finally, $S$ is called {\em the set of second category} if it is not meager. 

The following proposition lists some ``largeness'' properties of a set  $A\subset\R$  
which force it to be RAP. Denote by $\lam$ the Lebesgue measure on $\R$.
\begin{prop}[Classes of RAP sets]\label{prop:rap}

Let $A\subset\R$. Then $S$ is RAP if $S$ belongs to at least one of the following four classes:
\begin{align*}
\mtc E_1&=\{\text{$S\subset\R\mid S$  is Lebesgue measurable with  
   $0<\lam(S)\leq\infty$}\},\\
\mtc E_2&=\{\text{$S\subset\R\mid S$  is residual in some interval $X\subset\R$ of positive length}\},\\
\mtc E_3&=\{\text{$S\subset\R\mid S$  is winning in Schmidt's game}\},\\
&\h{19}\text{\small (several versions of Schmidt's games are possible, see \cite{Mc},\cite{Sc})},\\
\mtc E_4&=\{\text{$S\subset\R\mid S$  is Borel and not meager}\}
\end{align*}
\end{prop}
\begin{proof} For a set $S\in\mtc E_1$, one easily produces APs near any its 
Lebesgue density point. The argument for the classes  $E_2$ and $E_3$ is even easier
because residual subsets and the class $E_3$ are closed under finite 
(and even countable) intersections.

Finally, the sets $S\in\mtc E_4$ are RAP because $\mtc E_4\subset \mtc E_2$.
(A Borel subset $S\subset\R$ of second category must be residual in some subinterval,
see e.g.~Proposition 3.5.6 and Corollary 3.5.2 in \cite[page 108]{Sr}).
\end{proof}

Note that the the problems of finding finite or countable configurations $F$ in sets 
$S\subset\R$, under various ``largeness'' metric assumptions on $S$, has been considered 
by several mathematicians.

Following Kolountzakis \cite{Kol}, a set $F$ is called {\em universal} for a class $\mtc E$ of 
subsets of reals if $F\ll S$  for all $S\in \mtc E$. Henceforth $F\ll S$ means
that  $S$ contains an affine image of $F$, i.e. that $aF+b\subset S$, for some  
$a,b\in\R,\,a>0$. For example, $S$ is RAP iff $\{1,2,\ldots,n\}\ll S$ for all $n\geq1$; 
$S$ is FAP iff $\{1,2,3\}\not\!\ll S$.

Every finite subset of reals is universal for all the classes $\mtc E_k$,
$1\leq k\leq4$. 
Every bounded countable subset is universal for the classes $\mtc E_k$, $2\leq k\leq4$.

An old question of Erd\"os is whether there is an universal infinite set $F\subset\R$
for the class $\mtc E_1$ (of sets of positive measure).
The question is still open even though some families of countable sets $F$ are 
shown not to contain universal functions, see Kolountzakis \cite{Kol}, 
Paul and Laczkovich \cite{PL} and references there. 
In \cite{PL} an elegant combinatorial characterization of universal sets $F$
(for the class $\mtc E_1$) is given which reproduces earlier results in the subject.

Keleti \cite{Ke} constructed a compact set $A\subset\ui$ of Hausdorff dimension $1$ 
which is FAP; on the other hand, Laza and Pramanik in \cite{LP} showed that under certain 
assumptions (on the Fourier transform of supported measure) compact sets
of fractional dimension close to $1$ must contain $3$-term APs (i.e., cannot be FAP).
We refer to \cite{LP} for survey of related questions.

The central result of the paper, Theorem \ref{thm:central}, completely characterizes 
the topological (rather than metric) properties of a Borel set $S\subset\R$ which 
guarantee it to be RAP.
This theorem is an immediate consequence of the following proposition and the fact 
that the sets $S\in\mtc E_4$ must be RAP (Proposition \ref{prop:rap}).
\begin{prop}\label{prop:meag}
For every meager subset $C\subset\ui$, there is a map $\phi\in\H$, $\phi\colon\!\ui\to\ui$, 
such that $\phi(C)$  is FAP.
\end{prop}
A stronger version of Proposition \ref{prop:meag} (Proposition \ref{prop:meag2})
is presented and proved in the next section.
\section{Proofs of Propositions \ref{prop:meag} and \ref{prop:meag2}}\label{sec:pp1}
Denote by $\C$ the Banach space of continuous maps $f\colon \ui\to\R$
equipped with the norm 
\begin{equation}\label{eq:norm}
\|f\|=\|f\|_{\infty}=\max_{x\in\ui}\abs{f(x)}.
\end{equation}
Denote by $\F$ and $\hp$ the following subsets of $\C$:
\begin{align}
\F&=\{f\in\C\mid f \text{ is non-decreasing with }f(0)=0;\, f(1)=1\},\label{eq:F}\\
\hp&=\{f\in\F\mid f \text{ is injective}\}=\{f\in\H\mid f \text{ is increasing on }\ui\}\label{eq:H}.
\end{align} 

The set $\F$ is a closed subset of $\C$, while $\hp$ is residual in $\F$. 
(Indeed, 
\[
\Dst\hp=\hspm3\bigcap_{\substack{0<a<b<1\\a,b\in\Q}}F_{a,b}; \qquad
F_{a,b}=\big\{f\in\F\mid f(a)<f(b)\big\}
\]  
where $\Q$ stands for the set of rationals, and $F_{a,b}$  
are open dense subsets of $\F$).

The following proposition is a stronger version of Proposition \ref{prop:meag}.
\begin{prop}\label{prop:meag2}
Let \ $C\subset\ui$  be a meager subset. Then, for residual subset of $\phi\in\hp$,
the image $\phi(C)$  is FAP (has no $3$-term APs). 
\end{prop}
Since a meager set is a countable union of nowhere dense sets, it is enough to
prove the above proposition under the weaker assumption that $C$ is nowhere dense. 
Indeed, a meager set $C$ has a representation in the form
$
C=\cup_{i=k}^\infty C_k
$
where $C_k$ are nowhere dense. Then the unions $U_k=\cup_{i=1}^k C_i$ form 
a nested sequence of nowhere dense sets,  and $\phi(C)$  is FAP 
if all $\phi(U_k)$ are. 

Let
\begin{equation}\label{eq:hec}
\H_\eps(C)=\{\phi\in\hp\mid \phi(C) \text{ has no $3$-term APs of step }d\geq \eps\}.
\end{equation}

In the proof of Proposition \ref{prop:meag2} we need the following lemma. Its proof is provided
in the end of the next section.
\begin{lem}\label{lem:last}
Let\, $C\subset\ui$  be a nowhere dense subset and $\eps>0$. Then $\H_{\eps}(C)$ contains a dense
open subset of $\H^+$. In particular, $\H_{\eps}(C)$ is residual in $\H^+$.
\end{lem}
\begin{proof}[Proof of Proposition \ref{prop:meag2} assuming Lemma \ref{lem:last}] 
We may assume that $C$ is nowhere dense (see the sentence following 
Proposition~\ref{prop:meag2}). We may also assume that $C$ is compact 
(otherwise replacing $C$ by its closure $\bar C$). 

By Lemma \ref{lem:last}, each of the sets $\H_\eps(C)$, $\eps>0$, is residual in  $\H^+$.
It follows that the set  $\H_0(C)=\cap_{k=1}^\infty \H_{1/k}(C)$ is 
residual. It is also clear that,  for $\phi\in\H_0(C)$,  the images $\phi(C)$  are FAP.

This completes the proof of Proposition \ref{prop:meag2}.
\end{proof}


\section{Proof of Lemma \ref{lem:last}}
First we prepare some auxiliary results.
\begin{lem}\label{lem:ap31}
Let  $C\subset\ui$ be a nowhere dense set, let $f\in\hp$ and let $\eps>0$ be given.
Then there exists $g\in\hp$ such that $\|g-f\|<\eps$ and the set  $g(C)$  has no
$3$-term APs with step $d\geq\eps$.
\end{lem}
\begin{proof}
Without loss of generality, we assume that $\eps<1/2$. 
Pick an integer $r\geq3$  such that $r\eps>1$. 

Since $C$ is nowhere dense, so is $f(C)$, and one can select $r-1$ points\, 
$x_1,x_2,\ldots,x_{r-1}\in(0,1)\!\setminus f(\bar{C})$,
\[
0=x_0<x_1<x_2<\ldots<x_{r-1}<x_{r}=1,
\]
partitioning the unit interval into $r$ subintervals $X_{k}=(x_{k-1},x_k)$,
each shorter than $\eps$:
\[
0<\abs{X_k}=x_{k+1}-x_k<\eps \quad (1\leq k\leq r).
\]

Then one selects non-empty open subintervals $Y_{k}=(y^{-}_k,y_k^{+})\subset X_{k}$, $1\leq k\leq r$, in such
a way that the following four conditions are met:
\begin{align}
& (c1) \h2 f(C)\subset\bigcup_{k=1}^r\nolimits \bar Y_{k},\label{eq:cy}\\
& (c2) \h2 x_{k-1}<y_k^-<y_k^+<x_k  \text{ (i.e., }\,\bar Y_{k}\subset X_{k}), 
       \h2 \text{for } 2\leq k\leq r-1,\notag\\
& (c3) \h2 0=x_0=y_1^-< y_1^+<x_1, \ \text{ \small and }\notag\\
& (c4) \h2 x_{r-1}<y_r^-<y_r^+=x_r=1. \notag
\end{align}
That is, between $Y_j$ and $Y_{j+1}$ there exists $x_j \notin f(\bar{C})$ and 
$|Y_j|<|X_j|<\epsilon$ for all $j$.

Set $p_1=0$, $p_r=1$ and then select the $r-2$  points $p_k\in Y_{k}$, 
$2\leq k\leq r-1$,  so that the set $P=\{p_k\}_{k=1}^r$  contain no $3$-term APs. 
Then the sequence $(p_k)_1^r$ is strictly increasing, and
\[
\delta=\min_{1\leq m<n<k\leq r} \abs{p_m+p_k-2p_n}>0.
\]
Next, for  $1\leq k\leq r$,  we select open subintervals $Z_k\subset Y_k$,
each shorter than $\frac\del4$,  with $p_k\subset \bar Z_k$.

Define  $u\in\H$ to be the homeomorphism $\ui\to\ui$ which affinely contracts  
$\bar Y_k$  to $\bar Z_k$  and affinely expands the gaps between the intervals  
$\bar Y_k$  to fill it in. Note that
\begin{equation}\label{eq:cupyk}
\abs{u(x)-x}<\eps, \quad \text{ for }\, x\in\bigcup_{k=1}^r \bar Y_k,
\end{equation}
because  $x\in \bar Y_{k}$ implies $u(x)\in \bar Y_{k}$ and hence
$\abs{u(x)-x}\leq \abs{Y_k}<\abs{X_k}<\eps$. 

Since $u(x)-x$  is linear on each of the $(r-1)$ gaps between the intervals  $\bar Y_k$,
the inequality \eqref{eq:cupyk}  extends to the whole unit interval:
$
\|u(x)-x\|<\eps.
$

Define $g\in\H$ as the composition $g(x)=(u\circ f)x=u(f(x))$. Then
\[
\|g-f\|=\|u\circ f-f\|=\|u(x)-x\|<\eps.
\]

It remains to show that $g(C)$ has no $3$-term APs with step $d\geq\eps$.
In view of \eqref{eq:cy},
\[
\bigcup_{k=1}^r \bar Z_k= h(\bigcup_{k=1}^r \bar Y_k)\supset h(f(C))=g(C),
\]
so it would suffice to proof that  $\bigcup_{k=1}^r \bar Z_k$
has no $3$-term APs with step  $d\geq\eps$.

Assume to the contrary that such an AP exists, say $a_1, a_2, a_3$, with 
$d=a_2-a_1=a_3-a_2\geq\eps$.
Let $a_i\in\bar Z_{k_i}$, for $i=1,2,3$. These $k_i$ are uniquely determined, and since 
$|Z_{k_i}|<|X_{k_i}|<\eps\leq d$,  we have $k_1<k_2<k_3$.
Taking in account that $|a_i-p_{k_i}|\leq |Z_{k_i}|<\del/4$, we obtain
\begin{align*}
|a_1+a_3-2a_2|&\geq |p_{k_1}+p_{k_3}-2p_{k_2}|-\\
  &-(|a_1-p_{k_1}|+|a_3-p_{k_3}|+2|a_2-p_{k_2}|)>\del-4\cdot\tfrac\del4=0,
\end{align*}
a contradiction with the assumption that  $a_1, a_2, a_3$ forms an AP.
\end{proof}
\begin{cor}\label{cor:1}
Let  $C\subset\ui$ be a nowhere dense set. Then for all $\eps>0$, the sets 
$\H_{\eps}(C)$ (defined by \eqref{eq:hec}) are dense in $\hp$.
\end{cor}
\begin{proof} 
Note that the sets $\H_{\eps}(C)$ are monotone in $\eps>0$: $\H_{\eps_2}(C)\subset\H_{\eps_1}(C)$ 
if $0<\eps_2<\eps_1$.

By the previous lemma (Lemma \ref{lem:ap31}), all sets $\H_{\eps}(C)$ are $\eps$-dense.
Then, for a given $\eps>0$, the set $\H_{\eps}(C)$ is $\del$-dense for every positive $\del<\eps$
(because even the smaller set $\H_{\del}(C)\subset\H_{\eps}(C)$ is  $\del$-dense).
This argument completes the proof of Corollary \ref{cor:1}.
\end{proof} 
\begin{lem}\label{lem:open}
Let  $C\subset\ui$ be a compact nowhere dense set, let $g\in\H$ and let $\eps>0$ be given.
Assume that the set  $g(C)$  has no $3$-term APs with step $d\geq\eps$.
Then there exists a $\del>0$  such that for all $h\in\H$ such that $\|h-g\|<\del$
the sets $h(C)$  have no $3$-term APs with step exceeding $2\eps$.
\end{lem}
\begin{proof}
Let  
\[
M=\{(x_1,x_2,x_3)\in g(C)^3\mid x_2-x_1\geq \eps\, \text{ and }\, x_3-x_2\geq \eps\}.
\]
Then $M$ is compact, and $F\colon M\to\R$ defined by $F(x_1,x_2,x_3)=\abs{x_1+x_3-2x_2}$ 
assumes its minimum
\[
\gam=\min_{\bfx\in M} F(\bfx)>0
\]
which is positive because $g(C)$  has no $3$-term APs with step $d\geq\eps$.
Take $\del=\min(\eps/2,\gam/5)$.

Assume to the contrary that for some $h\in\H$ with $\|h-g\|<\del$,
the set $h(C)$ contains an AP with step $d'>2\eps$, i.e. that
there are $c_1,c_2,c_3\in C$ such that
\[
h(c_3)-h(c_2)=h(c_2)-h(c_1)>2\eps.
\]
Then, for both  $i=1,2$, we have
\[
g(c_{i+1})-g(c_i)> h(c_{i+1})-h(c_i)-2\del>2\eps-2\del\geq\eps,
\]
whence $(g(c_1),g(c_2),g(c_3))\in M$ and hence
\begin{align*}
\gam\leq F(g(c_1),g(c_2),g(c_3))&=\abs{g(c_1)+g(c_3)-2g(c_2)}\leq\\
 &\leq \abs{h(c_1)+h(c_3)-2h(c_2)}+4\del=0+4\del\leq\tfrac{4\gam}5<\gam,
\end{align*}
a contradiction.
\end{proof}
\begin{proof}[Proof of Lemma \em\ref{lem:last}]
It follows from Lemma \ref{lem:open} that there is an (intermediate) open subset $U\subset\hp$
such that
\[
\H_\eps(C)\subset U\subset \H_{2\eps}(C)\subset\hp.
\]
This set $U$  is dense in $\hp$ because its subset $\H_{\eps}(C)$ is (by Corollary~\ref{cor:1}).
Thus the set $\H_{2\eps}(C)$ contains an open dense subset $U\subset\hp$.
Since $\eps>0$ is arbitrary, the proof is complete.
\end{proof}


\end{document}